\documentclass[12pt]{article}

\usepackage{hyperref}
\usepackage{xcolor}
\usepackage{amssymb}
\usepackage{amsfonts}
\usepackage{amsthm}
\usepackage{amsmath}
\usepackage{float}
\usepackage{bigints}
\usepackage{fullpage}
\usepackage{mathtools}
\usepackage[utf8]{inputenc} 
\usepackage{cancel}
\usepackage{verbatim}

\newtheorem{theorem}{Theorem}[section]

\newtheorem{corollary}[theorem]{Corollary}

\newtheorem{definition}[theorem]{Definition}

\newtheorem{lemma}[theorem]{Lemma}

\newtheorem{problem}[theorem]{Problem}

\newtheorem{remark}[theorem]{Remark}

\renewcommand{\epsilon}{\varepsilon}

\newcommand{\B}{\mathbf B}
\newcommand{\cc}{\mathbf c}
\newcommand{\x}{\mathbf x}
\newcommand{\y}{\mathbf y}
\newcommand{\p}{\mathbf p}
\newcommand{\q}{\mathbf q}

\usepackage[pdftex]{graphicx}
\graphicspath{{C:/Users/Sam/Desktop/PaperFigures/}}
\DeclareGraphicsExtensions{.pdf}

\title{On basic $r$-ball polyhedra
\footnote{\hspace*{.35cm} 2010 Mathematics Subject Classification: (Primary) 05C10, 52C15, (Secondary) 05B40, 46B20.}}

\author{K\'{a}roly Bezdek\thanks{Partially supported by a Natural Sciences and 
Engineering Research Council of Canada Discovery Grant.}
}

\begin{document}

\maketitle

\begin{abstract}

This note introduces the class of basic $r$-ball polyhedra in the $d$-dimensional Euclidean space $\mathbb{E}^{d}$ for $d>1$ and $r>0$. We investigate their face structure and, for given integers $0\leq i\leq d-1$, $n\geq d+1\geq 3$ determine the maximal number of 
$i$-dimensional faces among all basic $r$-ball polyhedra in $\mathbb{E}^{d}$ with $n$ facets. In addition, we establish that for $d>2$, every basic $r$-ball polyhedron is globally rigid with respect to its inner dihedral angles.
\end{abstract}

\section{Basic $r$-ball polyhedra and their face lattices}\label{sec:intro}

Let $\mathbb{E}^{d}$, $d>1$ denote the $d$-dimensional Euclidean space, with the standard inner product $\langle\cdot ,\cdot\rangle$ and norm $\|\cdot\|$. Let $r>0$ be given. The closed (resp., open) Euclidean ball of radius $r$ centered at $\p\in\mathbb{E}^{d}$ is denoted by $\B^d[\p,r]:=\{\q\in\mathbb{E}^{d}\ |\  \|\p-\q\|\leq r\}$ (resp., $\B^d(\p,r):=\{\q\in\mathbb{E}^{d}\ |\  \|\p-\q\|< r\}$). As usual, ${\rm int}(\cdot)$, ${\rm bd}(\cdot)$, ${\rm cl}(\cdot)$, ${\rm conv}(\cdot)$, and ${\rm dim}(\cdot)$ refer to the interior, boundary, closure, convex hull, and dimension of the corresponding set in $\mathbb{E}^{d}$. Recall that an {\it $r$-ball polyhedron} (see for example, \cite{blnp}, \cite{KMP}, and \cite{MMO}) of $\mathbb{E}^{d}$ is an intersection with non-empty interior of finitely many closed balls of radius $r>0$ in $\mathbb{E}^{d} $, which are called the {\it generating balls} of the $r$-ball polyhedron.  Also, it is natural to assume that removing any of the generating balls yields the intersection of the remaining balls becoming a larger set. In other words, whenever we take an $r$-ball polyhedron we always assume that its generating balls form a {\it reduced family}. A very natural subfamily of $3$-dimensional $r$-ball polyhedra is formed by the so-called normal $r$-ball polyhedra, which have been introduced in \cite{Be2013a} (see also \cite{Be2013b}). In this note, we extend that definition to higher dimensions and introduce {\it basic $r$-ball polyhedra} in $\mathbb{E}^{d}$ for all $d>1$ and discuss some combinatorial and metric properties of them which are based on the underlying farthest point Voronoi diagrams. The details are as follows.

Let us start by recalling the notion of {\it farthest point Voronoi diagram} following \cite{Se91}. Let $C:=\{\cc_1,\dots ,\cc_n\}$ be a point set of $n>1$ (pairwise distinct) points in $\mathbb{E}^{d}$, $d>1$. We define the function $R_f$ that maps $\mathbb{E}^{d}$ into the powerset of $C$ by $$R_f(\x):=\{\cc_i\in C\ |\ \|\x-\cc_i\|\geq \|\x-\cc_j\|\ {\rm for\ all}\ 1\leq j\leq n\}.$$ The function $R_f$ defines an equivalence relation $\rho_f$ on $\mathbb{E}^{d}$, where $\x\rho_f\y$ holds if and only if $R_f(\x)=R_f(\y)$. The partition of $\mathbb{E}^{d}$ induced by $\rho_f$ is called the {\it farthest point Voronoi diagram} of $C$ denoted by $V_f(C)$. The equivalence classes of $V_f(C)$ are called the {\it faces} of $V_f(C)$ denoted by $V_S:=R_f^{-1}(S)$ for appropriate subsets $S$ of $C$. As usual we call the $0$-dimensional faces {\it vertices}, the $1$-dimensional faces {\it edges}, the $(d-1)$-dimensional faces {\it facets}, and the $d$-dimensional faces {\it farthest point Voronoi cells} of $V_f(C)$. It is well known (\cite{Se91}) that the (topological) closure of a face of $V_f(C)$ is a convex polyhedral set  (i.e., it is the intersection of a finite number of closed half-spaces) and the closures of the faces of $V_f(C)$ form a polyhedral complex (i.e., a finite family of convex polyhedral sets such that any face of a member of the family is again a member of the family, and if the intersection of two members is non-empty, then the intersection of them is a face of both members). 

Before we turn to the definition of basic $r$-ball polyhedra we need to recall another useful concept. Let  $C:=\{\cc_1,\dots ,\cc_n\}\subset\B^d(\p,r)$ for some $\p\in\mathbb{E}^{d}$, $d>1$ and $r>0$. Then let the {\it $r$-convex hull} ${\rm conv}_r(C)$ of $C$ be defined by ${\rm conv}_r(C):=\cap\{ \B^d[\x,r]\ |\ C\subset\B^d[\x,r]\}$. The following is a basic observation (that follows for example, from Corollary 3.4 of \cite{blnp} in a straightforward way):
$\B^d[\cc_1,r], \B^d[\cc_2,r], \dots , \B^d[\cc_n,r]$, $n>1$, $d>1$ form a reduced generating family for the $r$-ball polyhedron $\cap_{1\leq i\leq n}\B^d[\cc_i,r]$ with ${\rm int} \left( \cap_{1\leq i\leq n}\B^d[\cc_i,r]\right)\neq\emptyset$ if and only if the points of $C:=\{\cc_1, \cc_2, \dots ,\cc_n\}$ lie in {\it $r$-convex position}, i.e., $C\subset\B^d(\p,r)$ for some $\p\in\mathbb{E}^{d}$ and $r>0$ and $\cc_i\notin{\rm conv}_r(C\setminus\{\cc_i\})$ holds for all $1\leq i\leq n$. The following definition of {\it basic $r$-ball polyhedra} is an extension to higher dimensions of the analogue $3$-dimensional definition from \cite{Be2013a} and \cite{Be2013b}.

\begin{definition}\label{basic r-ball polyhedron}
Let $\mathbf{C}:={\rm conv}\left(\{\mathbf{c}_1,\mathbf{c}_2,\dots ,\mathbf{c}_n\}\right)$ be an arbitrary $d$-dimensional convex polytope with vertex set $C:=\{\mathbf{c}_1,\mathbf{c}_2,\dots ,\mathbf{c}_n\}$ in $\mathbb{E}^{d}$, where $n\geq d+1\geq 3$. Then, choose $r>0$ such that 
\begin{itemize}
\item the vertices of $\mathbf{C}$ lie in $r$-convex position,
and
\item every vertex of the farthest-point Voronoi diagram $V_f(C)$ of $C$ belongs to ${\rm int}\left(\mathbf{P}_{C, r}\right)$, where $\mathbf{P}_{C, r}:=\cap_{i=1}^{n}\mathbf{B}^d[\mathbf{c}_i, r]$.
\end{itemize}
Then $\mathbf{P}_{C, r}$ is called a {\rm $d$-dimensional basic $r$-ball polyhedron} having $r$ as its {\rm generating radius} and $\mathbf{C}={\rm conv} (C)$ as its {\it center-polytope}. If $V_S=R_f^{-1}(S)$ is a $k$-dimensional face of $V_f(C)$ for an appropriate subset $S$ of $C$ with $1\leq k\leq d$ such that $V_S\cap{\rm bd}(\mathbf{P}_{C, r})\neq\emptyset$, then $V_S\cap{\rm bd}(\mathbf{P}_{C, r})$ (resp., ${\rm cl}(V_S)\cap{\rm bd}(\mathbf{P}_{C, r})$) is called a {\rm $(k-1)$-dimensional face} (resp., {\rm closed $(k-1)$-dimensional face}) of the basic $r$-ball polyhedron $\mathbf{P}_{C, r}$. As usual we call the $0$-dimensional faces {\it vertices}, the $1$-dimensional faces {\it edges}, and the $(d-1)$-dimensional faces {\it facets} of $\mathbf{P}_{C, r}$. It follows that $\mathbf{P}_{C, r}$ possesses $n$ facets.
\end{definition}

\begin{remark}
Clearly, the closed faces of $\mathbf{P}_{C, r}$ introduced in Definition~\ref{basic r-ball polyhedron} together with $\emptyset$ and $\mathbf{P}_{C, r}$ (as degenerate faces) form a partially ordered set $L\left(\mathbf{P}_{C, r}\right)$ with respect to ordering by inclusion. As the intersection of two closed faces of $\mathbf{P}_{C, r}$ is either $\emptyset$ or a closed face of $\mathbf{P}_{C, r}$ therefore $L\left(\mathbf{P}_{C, r}\right)$ is a lattice - called the {\rm face lattice of $\mathbf{P}_{C, r}$} - with the meet operation $F\wedge F':=F\cap F'$ and the join operation $F\vee F':=\cap \{F''\in L\left(\mathbf{P}_{C, r}\right)\ |\ F\subseteq F''\ {\rm and}\ F'\subseteq F''\}$ for the closed faces $F,F'\in L\left(\mathbf{P}_{C, r}\right)$.
\end{remark}

Now, we are ready to state the promised combinatorial properties of the face lattices of $d$-dimensional basic $r$-ball polyhedra. Let $n>d>1$ and $\mathbf{C}(n,d)$ be the convex hull of any $n$ distinct points on the moment curve $(\tau, \tau^2, \dots , \tau^d)$, $-\infty<\tau<\infty$ in $\mathbb{E}^{d}$. Let $c_i(n,d)$ be the number of $i$-dimensional faces of $\mathbf{C}(n,d)$ for $0\leq i\leq d-1$. McMullen proved in \cite{Mc} that if $f_i$ is the number of $i$-dimensional faces of an arbitrary $d$-dimensional convex polytope with $n$ vertices, then $f_i\leq c_i(n,d)$, where $0\leq i\leq d-1$. The following theorem is an analogue of McMullen's upper bound theorem for basic $r$-ball polyhedra. 

\begin{theorem}\label{combinatorics of basic r-ball polyhedra}  Let $\mathbf{C}:={\rm conv}\left(\{\mathbf{c}_1,\mathbf{c}_2,\dots ,\mathbf{c}_n\}\right)$ be an arbitrary $d$-dimensional convex polytope with vertex set $C:=\{\mathbf{c}_1,\mathbf{c}_2,\dots ,\mathbf{c}_n\}$ in $\mathbb{E}^{d}$, where $n\geq d+1\geq 3$. If $f_{k-1}\left(\mathbf{P}_{C, r}\right)$ denotes the number of $(k-1)$-dimensional faces of a basic $r$-ball polyhedron $\mathbf{P}_{C, r}$ having $n$ facets, then $$f_{k-1}\left(\mathbf{P}_{C, r}\right) \leq c_{d-k}(n,d)$$ holds for all $0\leq k-1\leq d-1$.
\end{theorem}

\begin{remark}\label{equality}
It follows from the proof of Theorem~\ref{combinatorics of basic r-ball polyhedra} in a straightforward way that for the choice $\mathbf{C}:=\mathbf{C}(n,d)$ in Definition~\ref{basic r-ball polyhedron} the upper bounds of Theorem~\ref{combinatorics of basic r-ball polyhedra} are sharp.
\end{remark}

In order to state the promised metric property of $d$-dimensional basic $r$-ball polyhedra let us recall the following. According to the well-known theorem of Alexandrov \cite{Al} if $\mathbf{P}$ and $\mathbf{P'}$ are combinatorially equivalent convex polyhedra with equal corresponding face angles in $\mathbb{E}^{3}$, then $\mathbf{P}$ and $\mathbf{P'}$ have equal corresponding inner dihedral angles. As is well known Alexandrov theorem implies Cauchy's rigidity theorem (see for example, Theorem 26.8 and the discussion followed in \cite{Pa}) according to which if two convex polyhedra in $\mathbb{E}^{3}$ are combinatorially equivalent with the corresponding faces being congruent, then $\mathbf{P}$ is congruent to $\mathbf{P'}$. These theorems motivate and support the following analogue of them for $d$-dimensional basic $r$-ball polyhedra. For stating that result we need 

\begin{definition}\label{dihedral angle and global rigidity}
Let $\mathbf{P}_{C, r}$ be a basic $r$-ball-polyhedron of $\mathbb{E}^{d}$, $d\geq 3$. Moreover, let $V_S\cap{\rm bd}(\mathbf{P}_{C, r})$ be an arbitrary $(d-2)$-dimensional face of  $\mathbf{P}_{C, r}$ and let $\p\in V_S\cap{\rm bd}(\mathbf{P}_{C, r})$, where $S$ is a proper (unordered) pair of points from $C$. (See Parts (i) and (ii) of Theorem~\ref{basic properties-1}.) Then the two closed balls of radius $r$ centered at the points of $S$ are generating balls of $\mathbf{P}_{C, r}$ the boundaries of which contain $V_S\cap{\rm bd}(\mathbf{P}_{C, r})$. Now, the {\rm inner dihedral angle} assigned to $V_S\cap{\rm bd}(\mathbf{P}_{C, r})$ is the angular measure of the intersection of the two half-spaces supporting the two generating balls of radius $r$ at $\p$. Finally, we say that $\mathbf{P}_{C, r}$ is {\rm globally rigid with respect to its inner dihedral angles within the family of $d$-dimensional basic $r$-ball polyhedra} if the following holds. If $\mathbf{P}_{C', r}$ is another basic $r$-ball polyhedron of $\mathbb{E}^{d}$ whose face lattice $L\left(\mathbf{P}_{C', r}\right)$ is isomorphic to $L\left(\mathbf{P}_{C, r}\right)$ and whose inner dihedral angles are equal to the corresponding inner dihedral angles of $\mathbf{P}_{C, r}$, then $\mathbf{P}_{C', r}$ is congruent to  $\mathbf{P}_{C, r}$.
\end{definition}

\begin{theorem}\label{global rigidity of normal ball polyhedra}
Every basic $r$-ball polyhedron of $\mathbb{E}^{d}$, $d\geq 3$ is globally rigid with respect to its inner dihedral angles within the family of $d$-dimensional basic $r$-ball polyhedra.
\end{theorem}

The proofs of Theorems~\ref{combinatorics of basic r-ball polyhedra} and~\ref{global rigidity of normal ball polyhedra} are given in Section~\ref{combinatorial-metric}. Finally, Section~\ref{open-problems} raises relevant open questions.

\section{More on the faces of basic $r$-ball polyhedra}

First, we recall the following concept from \cite{blnp}. Let $\mathbf{P}=\cap_{i=1}^{n}\mathbf{B}^d[\mathbf{c}_i, r]$ be an $r$-ball polyhedron with the reduced system of generating balls $\mathbf{B}^d[\mathbf{c}_1, r], \dots , \mathbf{B}^d[\mathbf{c}_n, r]$ in $\mathbb{E}^{d}$, $d\geq 3, n>1, r>0$. We call $S^{d-1}(\mathbf{c}_i, r):={\rm bd} \left(\mathbf{B}^d[\mathbf{c}_i, r]\right)$ a {\it generating sphere} of $\mathbf{P}$, where $1\leq i\leq n$. The $l$-dimensional sphere $S^l(\mathbf{p}, s)$, $0\leq l\leq d-1$ with center $\mathbf{p}$ and radius $s\geq 0$ in $\mathbb{E}^{d}$ is called a {\it supporting sphere} of $\mathbf{P}$ if it is an intersection of some generating spheres of $\mathbf{P}$ such that $S^l(\mathbf{p}, s)\cap\mathbf{P}\neq\emptyset$. We say that $\mathbf{P}$ is a {\it standard $r$-ball polyhedron} if  for any supporting sphere $S^l(\mathbf{p}, s)$ of $\mathbf{P}$ the intersection $S^l(\mathbf{p}, s)\cap\mathbf{P}$ is homeomorphic to a closed Euclidean ball of some dimension. We shall prove in Theorem~\ref{basic properties-1} that for given $d$-dimensional convex polytope $\mathbf{C}:={\rm conv}\left(\{\mathbf{c}_1,\mathbf{c}_2,\dots ,\mathbf{c}_n\}\right)$ with vertices $C:=\{\mathbf{c}_1,\mathbf{c}_2,\dots ,\mathbf{c}_n\}$ in $\mathbb{E}^{d}$, where $n\geq d+1\geq 3$, if $r>0$ is sufficiently large, then the basic $r$-ball polyhedron $\mathbf{P}_{C, r}$ is also a standard $r$-ball polyhedron. On the other hand, it is not hard to see that there are standard $r$-ball polyhedra that are not basic ones. (For $3$-dimensional examples see page 91 in \cite{Be2013a}).

Next, let $\mathbf{P}_{C, r}=\cap_{i=1}^{n}\mathbf{B}^d[\mathbf{c}_i, r]$ be a basic $r$-ball polyhedron with $n$ facets and center-polytope $\mathbf{C}={\rm conv} (C)={\rm conv}\left(\{\mathbf{c}_1,\mathbf{c}_2,\dots ,\mathbf{c}_n\}\right)$ in $\mathbb{E}^{d}$, where $n\geq d+1\geq 3$. Based on Definition~\ref{basic r-ball polyhedron}, if $V_S=R_f^{-1}(S)$ is a $k$-dimensional face of the farthest point Voronoi diagram $V_f(C)$ for an appropriate subset $S$ of $C$ with $1\leq k\leq d$ such that $V_S\cap{\rm bd}(\mathbf{P}_{C, r})\neq\emptyset$, then $V_S\cap{\rm bd}(\mathbf{P}_{C, r})$ is called a $(k-1)$-dimensional face of the $r$-ball polyhedron $\mathbf{P}_{C, r}$. Moreover, every $(k-1)$-dimensional face of $\mathbf{P}_{C, r}$ is obtained in this way. Next, let us define
$$D_f(C):=\{{\rm conv}(S)\ |\ S\subseteq C\ {\rm and}\  V_S\neq\emptyset\}.$$
Recall (\cite{Se91}) that $D_f(C)$ is called the {\it farthest point Delaunay complex} of $C$, which is a polyhedral complex with the relative interiors of the members partitioning $\mathbf{C}={\rm conv} (C)$ in a face-to-face way. Moreover, each convex polytope of $D_f(C)$ is inscribed in a sphere. (For more properties of $D_f(C)$ see for example, \cite{AK00} or \cite{Se91}.) We need also
$$\partial D_f(C):=\{{\rm conv}(S)\ |\ {\rm conv}(S)\in D_f(C)  \ {\rm and}\  {\rm conv}(S)\subset {\rm bd}\left(\mathbf{C}\right)\},$$
which is a polyhedral complex as well with the relative interiors of the members partitioning ${\rm bd}\left(\mathbf{C}\right)$ in a face-to-face way. We call it the {\it farthest point boundary Delaunay complex} of $C$. From our point of view the most important properties of $D_f(C)$ and $\partial D_f(C)$ are the following. 

\begin{theorem}\label{basic properties-1}
Let $\mathbf{C}:={\rm conv}\left(\{\mathbf{c}_1,\mathbf{c}_2,\dots ,\mathbf{c}_n\}\right)$ be an arbitrary $d$-dimensional convex polytope with vertex set $C:=\{\mathbf{c}_1,\mathbf{c}_2,\dots ,\mathbf{c}_n\}$ in $\mathbb{E}^{d}$, where $n\geq d+1\geq 3$. Next, choose $r>0$ such that $\mathbf{P}_{C, r}=\cap_{i=1}^{n}\mathbf{B}^d[\mathbf{c}_i, r]$ is a basic $r$-ball polyhedron with $n$ facets having $\mathbf{C}={\rm conv} (C)$ as its center-polytope. Then the faces of $\mathbf{P}_{C, r}$ and the underlying farthest point Voronoi diagram $V_f(C)$ have the following properties.
\begin{itemize}
\item[(i)] $V_S=R_f^{-1}(S)$ is a $k$-dimensional face of $V_f(C)$ for an appropriate subset $S$ of $C$ if and only if ${\rm conv}(S)\in D_f(C)$ is a $(d-k)$-dimensional convex polytope, where $0\leq k\leq d$.
\item[(ii)] $V_S\cap{\rm bd}(\mathbf{P}_{C, r})$ is a $(k-1)$-dimensional face of $\mathbf{P}_{C, r}$ if and only if ${\rm conv}(S)\in \partial D_f(C)$ with ${\rm dim}\left({\rm conv}(S)\right)=d-k$, where $1\leq k\leq d$.
\item[(iii)] For given center-polytope $\mathbf{C}$ there exists $r_0>0$ such that for all $r\geq r_0$ the basic $r$-ball polyhedron $\mathbf{P}_{C, r}$ is also a standard $r$-ball polyhedron of $\mathbb{E}^{d}$ with every $(k-1)$-dimensional face $V_S\cap{\rm bd}(\mathbf{P}_{C, r})$ of $\mathbf{P}_{C, r}$ homeomorphic to a ball of dimension $k-1$ for $1\leq k\leq d$. 
\end{itemize}
\end{theorem}

\begin{proof}[Proof of Theorem~\ref{basic properties-1}:]
Part (i) is well known. (See for example, \cite{AK00} or \cite{Se91}.) In order to prove Part (ii) let us assume that $V_S\cap{\rm bd}(\mathbf{P}_{C, r})$ is a $(k-1)$-dimensional face of $\mathbf{P}_{C, r}$, where $1\leq k\leq d$. Here $V_S=R_f^{-1}(S)$ is a $k$-dimensional face of $V_f(C)$ for an appropriate subset $S$ of $C$, which must possess an unbounded edge say, $e$ (because, every vertex of $V_f(C)$ lies in ${\rm int}\left(\mathbf{P}_{C, r}\right)$ and $V_S\cap{\rm bd}(\mathbf{P}_{C, r})\neq\emptyset$). By Part (i) there exists a one-to-one map $\phi$ from the family of closed faces of $V_f(C)$ onto the family of all members of the polyhedral complex $D_f(C)$ that is inclusion reversing. So, by Part (i) we have that $\phi\left({\rm cl}(V_S)\right)={\rm conv}(S)\in D_f(C)$ with ${\rm dim}\left({\rm conv}(S)\right)=d-k$. Now, let us assume that ${\rm conv}(S)\in D_f(C)\setminus\partial D_f(C)$. We are going to derive a contradiction from it as follows. By Part (i) $\phi\left({\rm cl}(e)\right)\in D_f(C)$ is a $(d-1)$-dimensional convex polytope (inscribed in a sphere) that contains ${\rm conv}(S)$. As ${\rm conv}(S)\notin\partial D_f(C)$ therefore  $\phi\left({\rm cl}(e)\right)\in D_f(C)\setminus\partial D_f(C)$ implying that there exist two $d$-dimensional convex polytopes $\mathbf{Q}_1\in D_f(C)$ and $\mathbf{Q}_2\in D_f(C)$ such that $\phi\left({\rm cl}(e)\right)= \mathbf{Q}_1\cap \mathbf{Q}_2$. Again by Part (i) $\phi^{-1}(\mathbf{Q}_1)$ and  $\phi^{-1}(\mathbf{Q}_2)$ are two vertices of ${\rm cl(e)}=\phi^{-1}\left(\phi\left({\rm cl}(e)\right)\right)$ implying that $e$ is bounded, a contradiction. Thus, we have shown that if $V_S\cap{\rm bd}(\mathbf{P}_{C, r})$ is a $(k-1)$-dimensional face of $\mathbf{P}_{C, r}$, then ${\rm conv}(S)\in \partial D_f(C)$ with ${\rm dim}\left({\rm conv}(S)\right)=d-k$, where $1\leq k\leq d$. For the proof of the other direction let us take $S\subset C$ such that ${\rm conv}(S)\in \partial D_f(C)$ with ${\rm dim}\left({\rm conv}(S)\right)=d-k$, where $1\leq k\leq d$. By Part (i) $V_S=R_f^{-1}(S)$ is a $k$-dimensional face of $V_f(C)$. In order to show that $V_S\cap{\rm bd}(\mathbf{P}_{C, r})$ is a $(k-1)$-dimensional face of $\mathbf{P}_{C, r}$, we need to prove that $V_S\cap{\rm bd}(\mathbf{P}_{C, r})\neq\emptyset$. As ${\rm conv}(S)\in \partial D_f(C)$ therefore there exists a supporting hyperplane $H$ of $\mathbf{C}={\rm conv} (C)$ in $\mathbb{E}^{d}$ with $S\subset H$. As ${\rm conv}(S)$ is a convex polytope inscribed in a sphere therefore there exists a sequence of $(d-1)$-dimensional spheres $S^{d-1}(\p_n,r_n)$ for $n=1,2,\dots$ such that $\lim_{n\to +\infty} S^{d-1}(\p_n,r_n) =H$ with the closed $d$-dimensional ball bounded by $S^{d-1}(\p_n,r_n)$ containing $\mathbf{C}$ for all $n=1,2,\dots$. It follows that $\{\p_1,\p_2,\dots, \p_n,\dots\}\subset R_f^{-1}(S)$ implying that $V_S$ is an unbounded face of $V_f(C)$ and therefore $V_S\cap{\rm bd}(\mathbf{P}_{C, r})\neq\emptyset$. This completes the proof of Part (ii).

Finally, in order to prove Part (iii) let $\mathbf{P}_{C, r}$ be a $d$-dimensional basic $r$-ball polyhedron with center-polytope $\mathbf{C}={\rm conv} (C)$ having vertices $C:=\{\mathbf{c}_1,\mathbf{c}_2,\dots ,\mathbf{c}_n\}$ in $\mathbb{E}^{d}$, where $n\geq d+1\geq 3$ and $r>0$. Then let $S^l(\mathbf{p}, s):=S^{d-1}(\mathbf{c}_{i_1}, r)\cap S^{d-1}(\mathbf{c}_{i_2}, r)\cap\dots\cap S^{d-1}(\mathbf{c}_{i_k}, r)$ be an arbitrary $l$-dimensional supporting sphere of $\mathbf{P}_{C, r}$ with $F:=S^l(\mathbf{p}, s)\cap\mathbf{P}_{C, r}=S^l(\mathbf{p}, s)\cap{\rm bd}(\mathbf{P}_{C, r})  \neq\emptyset$, $1\leq i_1\leq i_2\leq \dots\leq i_k\leq n$, and $0\leq l\leq d-1$. As $S^{d-1}(\mathbf{c}_{i_j}, r)\cap{\rm bd}(\mathbf{P}_{C, r})={\rm cl}\left(R_f^{-1}(\mathbf{c}_{i_j})\right)\cap{\rm bd}(\mathbf{P}_{C, r})$ holds for all $1\leq j\leq k$ therefore
\begin{equation}\label{supporting-0}
\emptyset\neq F=S^l(\mathbf{p}, s)\cap{\rm bd}(\mathbf{P}_{C, r})={\rm bd}(\mathbf{P}_{C, r})\bigcap\left(\cap_{j=1}^{k}{\rm cl}\left(R_f^{-1}(\mathbf{c}_{i_j})\right) \right)
\end{equation}
\begin{equation}\label{supporting-1}
={\rm bd}(\mathbf{P}_{C, r})\bigcap {\rm cl}\left(R_f^{-1}(\{\mathbf{c}_{i_1}, \mathbf{c}_{i_2}, \dots , \mathbf{c}_{i_k}\})\right)
\end{equation} 
is a closed face of $\mathbf{P}_{C, r}$ in the sense of Definition~\ref{basic r-ball polyhedron}. Here by Part (ii) we get that ${\rm conv}(\{\mathbf{c}_{i_1}, \mathbf{c}_{i_2}, \dots , \mathbf{c}_{i_k}\})\in \partial D_f(C)$ and
\begin{equation}\label{supporting-2}
{\rm dim}(F)=d-{\rm dim}\left({\rm aff}(\{\mathbf{c}_{i_1}, \mathbf{c}_{i_2}, \dots , \mathbf{c}_{i_k}\})\right)-1=l.
\end{equation}
Next, we recall the following concept. If $P$ is a convex polyhedral set in $\mathbb{E}^{d}$, then the {\it recession cone} ${\rm rec}(P)$ of $P$ is defined by ${\rm rec}(P):=\{\mathbf{y}\in\mathbb{E}^{d}\ |\ \mathbf{x}+\lambda\mathbf{y}\in P{\rm\ for\ all\ }\lambda\geq 0{\rm\ and\ }\mathbf{x}\in P\}$. We observe in Lemma~\ref{rec1} that the faces of the farthest point Voronoi diagram $V_f(C)$ have top dimensional recession cones relative to the affine hulls of the faces and note that this property is the driving force of the proof of Lemma~\ref{rec2}.
\begin{lemma}\label{rec1}
The recession cone ${\rm rec}\left( {\rm cl}\left(R_f^{-1}(\{\mathbf{c}_{i_1}, \mathbf{c}_{i_2}, \dots , \mathbf{c}_{i_k}\})\right)\right)$ of the $(l+1)$-dimensional closed face ${\rm cl}\left(R_f^{-1}(\{\mathbf{c}_{i_1}, \mathbf{c}_{i_2}, \dots , \mathbf{c}_{i_k}\})\right)$ of $V_f(C)$ is a convex polyhedral cone which is pointed (i.e., contains no lines) furthermore, its dimension is equal to $$d-{\rm dim}\left({\rm aff}(\{\mathbf{c}_{i_1}, \mathbf{c}_{i_2}, \dots , \mathbf{c}_{i_k}\})\right)=l+1.$$
\end{lemma}
\begin{proof}[Proof of Lemma~\ref{rec1}:]
Let $S:=\{\mathbf{c}_{i_1}, \mathbf{c}_{i_2}, \dots , \mathbf{c}_{i_k}\}$. Then $V_S=R_f^{-1}(S)$ is an $(l+1)$-dimensional face of $V_f(C)$ generating the $l$-dimensional face $F= {\rm bd}(\mathbf{P}_{C, r})\bigcap {\rm cl}\left(R_f^{-1}(S)\right)$ for $\mathbf{P}_{C, r}$. Thus, Part (ii) implies that ${\rm conv}(S)\in \partial D_f(C)$ is a $(d-l-1)$-dimensional face of the center-polytope $\mathbf{C}$. Now, let $\mathbf{s}\in{\rm relint}\left({\rm conv}(S)\right)$ be an arbitrary relative interior point of ${\rm conv}(S)$. Then the {\it normal cone} $N_{\mathbf{C}} ({\rm conv}(S))$ of the center-polytope $\mathbf{C}$ at the face ${\rm conv}(S)$ is defined by 
$$N_{\mathbf{C}} ({\rm conv}(S)):=\{\mathbf{y}\in\mathbb{E}^{d}\ |\ \langle \mathbf{x}-\mathbf{s}, \mathbf{y}\rangle\leq 0, {\rm\ for\ all}\ \mathbf{x}\in \mathbf{C}\}. $$
Now, we are ready to recall that according to Section 3 of \cite{GoMaTo}
\begin{equation}\label{GMT}
{\rm rec}(V_S)=-N_{\mathbf{C}} ({\rm conv}(S)) {\rm \ with\ }{\rm dim}(V_S)={\rm dim}(N_{\mathbf{C}} ({\rm conv}(S))).
\end{equation}
As $N_{\mathbf{C}} ({\rm conv}(S))$ is clearly a pointed $(l+1)$-dimensional convex polyhedral cone in $\mathbb{E}^{d}$, (\ref{GMT}) completes the proof of Lemma~\ref{rec1}.
\end{proof}
%For the next statement we need to recall the following concept. If $S^l(\mathbf{p}_0, s_0)$ is an $l$-dimensional sphere in $\mathbb{E}^{d}$ with $0\leq l\leq d-1$, then a set $F_0\subsetneq S^l(\mathbf{p}_0, s_0)$ is called {\it spherically convex} if it is contained in an open hemisphere of $S^l(\mathbf{p}_0, s_0)$ and for every $\mathbf{x}, \mathbf{y}\in F_0$ the shorter great-circular arc of $S^l(\mathbf{p}_0, s_0)$ connecting $\mathbf{x}$ with $\mathbf{y}$ is in $F_0$.

\begin{lemma}\label{rec2}
Let $\mathbf{p}^*\in\mathbb{E}^{l+1}, l\geq 0$ and let $\mathbf{Q}$ be an unbounded closed $(l+1)$-dimensional convex polyhedral set in $\mathbb{E}^{l+1}$ with an $(l+1)$-dimensional pointed recession cone. Then there exists $s_0>0$ with the following property: for all $s^*\geq s_0$ the vertices of $\mathbf{Q}$ are contained in $\mathbf{B}^{l+1}(\mathbf{p}^*, s^*)$ with ${\rm bd}\left(\mathbf{B}^{l+1}(\mathbf{p}^*, s^*)\right)=:S^l(\mathbf{p}^*, s^*)$ such that $S^l(\mathbf{p}^*, s^*)\cap\mathbf{Q}$ is homeomorphic to a closed $l$-dimensional Euclidean ball. 
\end{lemma} 
\begin{proof}[Proof of Lemma~\ref{rec2}:]
Let $\{F_i(\mathbf{Q})\ |\ i\in I_{\mathbf{Q}}\}$ be the family of facets of $\mathbf{Q}$ and let $H_i(\mathbf{Q})$ be the hyperplane of $F_i(\mathbf{Q})$ in $\mathbb{E}^{l+1}$ bounding the closed half-space $H_i^+(\mathbf{Q})$ which contains $\mathbf{Q}$ for $i\in I_{\mathbf{Q}}$. As ${\rm dim}\left({\rm rec}(\mathbf{Q})\right)=l+1$, the Theorem of \cite{Be92} implies in a straightforward way that $\mathbf{Q}$ can be illuminated by a single point-source in $\mathbb{E}^{l+1}$, i.e., there exists a point $\mathbf{q}\in \mathbb{E}^{l+1}$ such that  $\mathbf{q}\notin H_i^+(\mathbf{Q})$ for all $i\in I_{\mathbf{Q}}$. Now, let $s_0>0$ such that the vertices of $\mathbf{Q}$ are contained in $\mathbf{B}^{l+1}(\mathbf{p}^*, s_0)$ and also $\mathbf{q}\in \mathbf{B}^{l+1}(\mathbf{p}^*, s_0)$. Furthermore, let $\mathbf{q}_0\in{\rm int}(\mathbf{Q})\cap \mathbf{B}^{l+1}(\mathbf{p}^*, s_0)$. Finally, let $s^*\geq s_0$. Clearly, $\mathbf{B}^{l+1}(\mathbf{p}^*, s^*)$ contains the vertices of $\mathbf{Q}$ as well as $\mathbf{q}_0$ and $\mathbf{q}$, implying that the half-line starting at $\mathbf{q}_0$ and passing through $\mathbf{q}$ intersects $S^l(\mathbf{p}^*, s^*)$ say, in $\mathbf{q}^*$. It follows that
\begin{equation}\label{illumination by point}
\mathbf{q}^*\notin H_i^+(\mathbf{Q})\ {\rm for\ all\ }i\in I_{\mathbf{Q}},
\end{equation}
i.e., $\mathbf{q}^*\in S^l(\mathbf{p}^*, s^*)$ as a point-source illuminates $\mathbf{Q}$ in $\mathbb{E}^{l+1}$. Let the closed $l$-dimensional spherical cap $C_i(\mathbf{p}^*, s^*)$ be defined by $C_i(\mathbf{p}^*, s^*):=H_i^+(\mathbf{Q})\cap S^l(\mathbf{p}^*, s^*)$ for $i\in I_{\mathbf{Q}}$. It follows from (\ref{illumination by point}) that
\begin{equation}\label{stereographic-projection}
\mathbf{q}^*\notin C_i(\mathbf{p}^*, s^*)\ {\rm for\ all\ }i\in I_{\mathbf{Q}}{\rm \ and\ }S^l(\mathbf{p}^*, s^*)\cap\mathbf{Q}=\bigcap_{i\in I_{\mathbf{Q}}}C_i(\mathbf{p}^*, s^*).
\end{equation} 
Now, let $\pi_{\mathbf{q}^*}: S^l(\mathbf{p}^*, s^*)\to H^*$ be the stereographic projection with center (or pole) at $\mathbf{q}^*$ mapping $S^l(\mathbf{p}^*, s^*)\setminus\{\mathbf{q}^*\}$ onto the hyperplane $H^*$ of $\mathbb{E}^{l+1}$, where $H^*$ is tangent to $S^l(\mathbf{p}^*, s^*)$ at a point diametrically opposite to $\mathbf{q}^*$. Thus, (\ref{stereographic-projection}) implies in a straightforward way that $\pi_{\mathbf{q}^*}\left(C_i(\mathbf{p}^*, s^*)\right)$ is a closed $l$-dimensional Euclidean ball in $H^*$ for all $i\in I_{\mathbf{Q}}$ and therefore
\begin{equation}
\pi_{\mathbf{q}^*}\left(S^l(\mathbf{p}^*, s^*)\cap\mathbf{Q}\right)=\bigcap_{i\in I_{\mathbf{Q}}}\pi_{\mathbf{q}^*}\left(C_i(\mathbf{p}^*, s^*)\right)
\end{equation}
is an $l$-dimensional compact convex set with non-empty interior in $H^*$, finishing the proof of Lemma~\ref{rec2}.
\end{proof}
\noindent Finally, applying Lemma~\ref{rec2} to any $l$-dimensional face $F$ of $\mathbf{P}_{C, r}$ with a representation given by (\ref{supporting-0}), (\ref{supporting-1}) and (\ref{supporting-2}) for $0\leq l\leq d-1$, yields that $F$ is  homeomorphic to a closed $l$-dimensional Euclidean ball. Since there are only finitely many $F$'s to consider the existence of $r_0$ in Part (iii) of Theorem~\ref{basic properties-1} is also clear. This completes the proof of Theorem~\ref{basic properties-1}.
\end{proof}

The proof of Part (iii) of Theorem~\ref{basic properties-1} and  the Euler-Poincar\'e formula for standard $r$-ball polyhedra (see Corollary 6.10 in \cite{blnp}) imply the following statement in a straighforward way.
\begin{corollary}
For given center-polytope $\mathbf{C}$ there exists $r_0>0$ such that for all $r\geq r_0$ the basic $r$-ball polyhedron $\mathbf{P}_{C, r}$ is also a standard $r$-ball polyhedron in $\mathbb{E}^{d}$ possessing the following Euler-Poincar\'e formula:
$$1+(-1)^{d+1}=\sum_{i=0}^{d-1}(-1)^if_i(\mathbf{P}_{C, r}).$$
 \end{corollary}

\section{Proofs of Theorems~\ref{combinatorics of basic r-ball polyhedra} and~\ref{global rigidity of normal ball polyhedra}}\label{combinatorial-metric}

\begin{proof}[Proof of Theorem ~\ref{combinatorics of basic r-ball polyhedra} :]
Let $\mathbf{C}:={\rm conv}\left(\{\mathbf{c}_1,\mathbf{c}_2,\dots ,\mathbf{c}_n\}\right)$ be an arbitrary $d$-dimensional convex polytope with vertex set $C:=\{\mathbf{c}_1,\mathbf{c}_2,\dots ,\mathbf{c}_n\}$ in $\mathbb{E}^{d}$, where $n\geq d+1\geq 3$. Next, choose $r>0$ such that $\mathbf{P}_{C, r}=\cap_{i=1}^{n}\mathbf{B}^d[\mathbf{c}_i, r]$ is a basic $r$-ball polyhedron with $n$ facets having $\mathbf{C}={\rm conv} (C)$ as its center-polytope. According to Part (ii) of Theorem~\ref{basic properties-1},  $V_S\cap{\rm bd}(\mathbf{P}_{C, r})$ is a $(k-1)$-dimensional face of $\mathbf{P}_{C, r}$ if and only if ${\rm conv}(S)\in \partial D_f(C)$ with ${\rm dim}\left({\rm conv}(S)\right)=d-k$, where $1\leq k\leq d$. Thus, 
the number $f_{k-1}\left(\mathbf{P}_{C, r}\right)$ of $(k-1)$-dimensional faces of $\mathbf{P}_{C, r}$ is equal to the number of $(d-k)$-dimensional convex cells ${\rm conv}(S)\in \partial D_f(C)$ of the farthest point boundary Delaunay complex $\partial D_f(C)$. We note that every $(d-k)$-dimensional convex cell ${\rm conv}(S)\in \partial D_f(C)$ is a subset of a $(d-k)$-dimensional face of the center-polytope $\mathbf{C}$ and $\partial D_f(C)$ is a polyhedral complex with $n$ vertices partitioning ${\rm bd}(\mathbf{C})$. Clearly, $\partial D_f(C)$ generates a spherical polyhedral complex of a $(d-1)$-dimensional sphere via radial projection from an interior point of $\mathbf{C}$ onto a $(d-1)$-dimensional sphere centered at the chosen interior point.  From Section 2.5 in \cite{McSh} on pulling the vertices of a polytope, we recall that in discussing the problem of maximizing the number of faces of a polytope, it is sufficient to restrict attention to simplicial polytopes and, in a similar way, for spherical complexes, we need only consider simplicial ones (see also Section 4.1 in \cite{McSh}). As a result, we finish by recalling Stanley's Upper Bound Theorem (\cite{Sta}) according to which the number of $(d-k)$-dimensional faces in an arbitrary triangulation of a $(d-1)$-dimensional sphere with $n$ vertices is at most $c_{d-k}(n,d)$, where $0\leq k-1\leq d-1$. Thus, $f_{k-1}\left(\mathbf{P}_{C, r}\right) \leq c_{d-k}(n,d)$ holds for all $0\leq k-1\leq d-1$. This completes the proof of Theorem~\ref{combinatorics of basic r-ball polyhedra}.
\end{proof}
\begin{proof}[Proof of Theorem~\ref{global rigidity of normal ball polyhedra}:]
Let  $\mathbf{P}_{C, r}$ (resp., $\mathbf{P}_{C', r}$) be a basic $r$-ball polyhedron with center-polytope $\mathbf{C}$ (resp., $\mathbf{C}'$) having vertex set $C:=\{\mathbf{c}_1,\mathbf{c}_2,\dots ,\mathbf{c}_n\}$ (resp., $C':=\{\mathbf{c}'_1,\mathbf{c}'_2,\dots ,\mathbf{c}'_n\}$) in $\mathbb{E}^{d}$, $d\geq 3$ such that the face lattices $L\left(\mathbf{P}_{C, r}\right)$ and $L\left(\mathbf{P}_{C', r}\right)$ are isomorphic and the inner dihedral angles of $\mathbf{P}_{C, r}$ are equal to the corresponding inner dihedral angles of $\mathbf{P}_{C', r}$. As the face lattices $L\left(\mathbf{P}_{C, r}\right)$ and $L\left(\mathbf{P}_{C', r}\right)$ are isomorphic, Parts (i) and (ii) of Theorem~\ref{basic properties-1} imply in a straightforward way that the farthest point boundary Delaunay complexes $\partial D_f(C)$ and $\partial D_f(C')$ are isomorphic with the relative interiors of the members partitioning ${\rm bd}(\mathbf{C})={\rm bd}\left({\rm conv}(C)\right)$ and ${\rm bd}(\mathbf{C}')={\rm bd}\left({\rm conv}(C')\right)$) in a face-to-face way. The assumption that the corresponding dihedral angles of $\mathbf{P}_{C, r}$ and $\mathbf{P}_{C', r}$ are equal means via Part (ii) of Theorem~\ref{basic properties-1} that the corresponding edges of the polyhedral complexes $\partial D_f(C)$ and $\partial D_f(C')$ are equal. As every member of the polyhedral complexes $\partial D_f(C)$ and $\partial D_f(C)$ is a convex polytope inscribed in a sphere, therefore the following statement proves that the corresponding members of $\partial D_f(C)$ and $\partial D_f(C')$ are congruent.
\begin{lemma}\label{special Cauchy}
Let $\mathbf{Q}$ and $\mathbf{Q}'$ be $d$-dimensional convex polytopes in $\mathbb{E}^{d}$, $d\geq 2$ such that their face lattices are isomorphic with the corresponding edges having equal length. If every $k$-face ($2\leq k\leq d$) of $\mathbf{Q}$ and $\mathbf{Q}'$ is inscribed in a $(k-1)$-sphere, then $\mathbf{Q}$ and $\mathbf{Q}'$ are congruent.
\end{lemma}
\begin{proof}[Proof of Lemma~\ref{special Cauchy}:]
One can prove the statement by induction on $d\geq 2$ as follows. For $d=2$ one needs to show that a convex polygon which can be inscribed in a circle in $\mathbb{E}^{2}$, is uniquely determined, up to congruence, by the lengths and cyclic order of its sides. We leave the proof of this elementary fact to the reader. Then we assume that for any two isomorphic "inscribed-type" $(d-1)$-dimensional polytopes with equal corresponding edge lengths, the polytopes are congruent. Next, consider any $(d-1)$-dimensional face (facet) $F$ of $\mathbf{Q}$ and its corresponding facet $F'$ of  $\mathbf{Q}'$. Applying the inductive hypothesis, we conclude that $F$ and $F'$ are congruent. We now have two $d$-dimensional convex polytopes $\mathbf{Q}$ and $\mathbf{Q}'$ in $\mathbb{E}^{d}$ such that their face lattices are isomorphic and all their corresponding facets are congruent. Finally, we recall the theorem of Alexandrov (see Chapter 3 in \cite{Al}) stating that if two $d$-dimensional convex polytopes in $\mathbb{E}^{d}$ have isomorphic face lattices and congruent corresponding facets, then the polytopes themselves are congruent in $\mathbb{E}^{d}$. Since the conditions of this theorem of Alexandrov hold for $\mathbf{Q}$ and $\mathbf{Q}'$, we conclude that $\mathbf{Q}$ must be congruent to $\mathbf{Q}'$, finishing the inductive proof of Lemma~\ref{special Cauchy}. \end{proof}
As an immediate corollary of Lemma~\ref{special Cauchy} we get that the farthest point boundary Delaunay complexes $\partial D_f(C)$ and $\partial D_f(C')$ are isomorphic with the relative interiors of the members partitioning ${\rm bd}(\mathbf{C})={\rm bd}\left({\rm conv}(C)\right)$ and ${\rm bd}(\mathbf{C}')={\rm bd}\left({\rm conv}(C')\right)$) in a face-to-face way such that the corresponding members of $\partial D_f(C)$ and $\partial D_f(C')$ are congruent. Now, we recall Alexandrov's Uniqueness Theorem for Polytopes from Chapter 3 of \cite{Al} as follows: Let $\mathbf{Q}$ and $\mathbf{Q}'$ be $d$-dimensional convex polytopes in $\mathbb{E}^{d}$, $d\geq 3$. If there exists a homeomorphism $f:{\rm bd}(\mathbf{Q}) \to {\rm bd}(\mathbf{Q}')$ between their boundaries such that $f$ is an intrinsic isometry, then $\mathbf{Q}$ and $\mathbf{Q}'$ are congruent. As it is not hard to see that the conditions of this theorem of Alexandrov hold for ${\rm bd}(\mathbf{C})$ and ${\rm bd}(\mathbf{C}')$ via the polyhedral complexes $\partial D_f(C)$ and $\partial D_f(C')$, we conclude that $\mathbf{C}$ and $\mathbf{C}'$ are congruent and therefore $\mathbf{P}_{C, r}$ and $\mathbf{P}_{C', r}$ are congruent as well. This completes the proof of Theorem~\ref{global rigidity of normal ball polyhedra}.
\end{proof}

\section{Concluding Remarks}\label{open-problems}

\begin{problem}\label{upper bounding the number of faces}
Prove or disprove that Theorems~\ref{combinatorics of basic r-ball polyhedra} and~\ref{global rigidity of normal ball polyhedra} extend to standard $r$-ball polyhedra in $\mathbb{E}^{d}$, $d\geq 3$.
\end{problem}

\begin{remark}\label{problem for d=3}
As the Euler-Poincar\'e formula holds for standard $r$-ball polyhedra in $\mathbb{E}^{d}$, $d\geq 3$, therefore an easy computation implies that Theorems~\ref{combinatorics of basic r-ball polyhedra} extends to standard $r$-ball polyhedra of $\mathbb{E}^{3}$, i.e., an arbitrary standard $r$-ball polyhedron with $n$ facets in $\mathbb{E}^{3}$ has at most $c_2(n,3)=2n-4$ vertices and $c_1(n,3)=3n-6$ edges. Furthermore, recall that a $3$-dimensional standard $r$-ball polyhedron is called simple if at every vertex exactly $3$ edges meet. The main result of \cite{BeNa} proves a local version of Theorem~\ref{global rigidity of normal ball polyhedra} for $3$-dimensional standard $r$-ball polyhedra stating that every simple and standard $r$-ball ball-polyhedron of $\mathbb{E}^{3}$ is locally rigid with respect to its inner dihedral angles.
\end{remark}

%Recall that a vertex of an $r$-ball-polyhedron in $\mathbb{E}^{d}$ is a boundary point that belongs to at least $d$ generating balls of the $r$-ball polyhedron. The following question is partly a special case of Problem~\ref{upper bounding the number of faces} and partly a broader one.

%\begin{problem}\label{upper bounding the number of vertices}
%Prove or disprove that the number of vertices of an arbitrary $d$-dimensional $r$-ball polyhedron generated by $n>1$ closed balls of radius $r>0$ is at most $c_{d-1}(n,d)$ for $d\geq 3$.
%\end{problem}

\bigskip

%\normalsize

\noindent K\'aroly Bezdek \\
\small{Department of Mathematics and Statistics, University of Calgary, Canada}\\
\small{Department of Mathematics, University of Pannonia, Veszpr\'em, Hungary\\
\small{E-mail: \texttt{bezdek@math.ucalgary.ca}}


\bigskip

\small
\begin{thebibliography}{GGM}

\bibitem{Al}
A. D. Alexandrov, Convex polyhedra (English translation by N. S. Dairbekov, S. S. Kutateladze and A. B. Sossinsky), \emph{Springer-Berlin}, 2005.

%\bibitem{ALN}
%S. M. Almohammad, Z. L\'angi, and M. Nasz\'odi, An analogue of a theorem of Steinitz for ball polyhedra in $\R^3$, \emph{Aequationes Math.} \textbf{96/2} (2022), 403--415.

%\bibitem{At}
%Ch. A. Athanasiadis, On the graph connectivity of skeleta of convex polytopes,  \emph{Discrete Comput. Geom.} \textbf{42/2} (2009), 155--165.


\bibitem{AK00}F.~Aurenhammer and R.~Klein, Voronoi diagrams, \emph{Handbook of computational geometry, North-Holland, Amsterdam}, 2000, 201--290.

\bibitem{Be92}
K. Bezdek, On the illumination of unbounded closed convex sets, \emph{Israel J. Math.} \textbf{80/1-2} (1992), 87--96. 

\bibitem{blnp} K. Bezdek, Zs. L\'angi, M. Nasz\'odi, and P. Papez,  Ball-polyhedra,  \emph{Discrete Comput. Geom.} \textbf{38/2} (2007), 201--230.

\bibitem{BeNa}K. Bezdek and M. Nasz\'odi, Rigid ball-polyhedra in Euclidean 3-space, \emph{Discrete Comput. Geom.} \textbf{49/2} (2013), 189--199. 

\bibitem{Be2013a}
K. Bezdek, Lectures on sphere arrangements - the discrete geometric side, \emph{Fields Institute Monographs} \textbf{32} Springer, New York; Fields Institute for Research in Mathematical Sciences, Toronto, ON, 2013.

\bibitem{Be2013b}
K. Bezdek, Globally rigid ball-polyhedra in Euclidean 3-space, \emph{Transactions on Comput. Sci. in Lecture Notes in Comput. Sci.} \textbf{8110} (2013), Springer, Heidelberg, 158--169. 

\bibitem{GoMaTo}
M. A. Goberna, J. E. Mart\'inez-Legaz, M. I. Todorov, On farthest Voronoi cells, \emph{Linear Algebra Appl.} \textbf{583} (2019), 306--322.
 


%\bibitem{EKS83}
%H.~Edelsbrunner, D. G. Kirkpatrick, and R. Seidel, On the shape of a set of points in the plane, \emph{IEEE Trans. Inform. Theory} \textbf{29/4} (1983), 551--559.


%\bibitem{Gr}
%B. Gr\"unbaum, Convex polytopes, \emph{Graduate Texts in Mathematics} \textbf{221}, Springer-Verlag, New York, 2003. 

\bibitem{KMP}
Y. S. Kupitz, H. Martini, and M. A. Perles, Ball polytopes and the V\'azsonyi problem, \emph{Acta Math. Hungar.}  \textbf{126/1-2} (2010), 99--163.

\bibitem{Mc}
P. McMullen, The maximum numbers of faces of a convex polytope, \emph{Mathematika} \textbf{17} (1970), 179--184.

\bibitem{McSh}
P. McMullen and G. C. Shephard, Convex polytopes and the upper bound conjecture, London Math. Soc. Lecture Note Ser., \emph{Cambridge University Press, London-New York}, 1971.

\bibitem{MMO}
H. Martini, L. Montejano, and D. Oliveros, Bodies of constant width - An introduction to convex geometry with applications, \emph{Birkh\"auser-Springer}, 2019.

\bibitem{Pa}
I. Pak, Lectures on discrete and polyhedral geometry,  https://www.math.ucla.edu/pak/book.htm, 1--442.


\bibitem{Se91}
R. Seidel, Exact upper bounds for the number of faces in d-dimensional Voronoi diagrams, \emph{DIMACS Ser. Discrete Math. Th. Comput. Sci., Amer. Math. Soc., Appl. Geom. Discrete Math.} \textbf{4} (1991), 517--529.

\bibitem{Sta}
R. P. Stanley, The upper bound conjecture and Cohen-Macaulay rings, \emph{Studies in Appl. Math.} \textbf{54/2} (1975), 135--142.


%\bibitem{St}
%E. Steinitz,  Polyeder und Raumeinteilungen, \emph{Encyclop\"adie der mathematischen Wissenschaften} \textbf{IIIAB12} (1916), 1--139.



%\bibitem{Zi}
%G. M. Ziegler, Lectures on polytopes, \emph{Graduate Texts in Mathematics} \textbf{152}, Springer-Verlag, New York, 1995.

\end{thebibliography}
\end{document}